\documentclass[intlimits]{amsart}

\usepackage[all,hyperref]{bi-discrete} 
\usepackage[backend=biber,maxbibnames=5,maxalphanames=5,style=alphabetic,bibencoding=utf8,giveninits,url=false,isbn=false]{biblatex}
\AtBeginBibliography{\small}

\newtheorem{axiom}[theorem]{Axiom}

\providecommand{\dx}{\, \mathrm{d}x}
\providecommand{\dy}{\, \mathrm{d}y}

\providecommand{\tria}{\mathcal{T}}

\providecommand{\meshes}{\mathbb{T}}
\providecommand{\osc}{{\ensuremath{\mathrm{osc}}}}
\providecommand{\err}{{\ensuremath{\mathrm{err}}}}
\providecommand{\CBDD}{C_{\textsc{\scriptsize BDD}}}

\begin{document} 
	
\title[Optimal AFEM convergence for any bulk parameter]{Optimal AFEM convergence for any bulk parameter smaller than one}

\author[L.\ Diening]{Lars Diening}
\author[J.\ Storn]{Johannes Storn}

\address[L.\ Diening]{Department of Mathematics, Bielefeld University, Postfach 10 01 31, 33501 Bielefeld, Germany}
\email{lars.diening@uni-bielefeld.de}
\address[J.\ Storn]{Faculty of Mathematics \& Computer Science, Institute of Mathematics, Leipzig University, Augustusplatz 10, 04109 Leipzig, Germany}
\email{johannes.storn@uni-leipzig.de}

\keywords{AFEM, optimal convergence, bulk parameter, Ritz projection}

\subjclass[2020]{65N30,  
65N15,  
65N12,  
41A25
}


\begin{abstract}
  We prove rate-optimal convergence of the standard adaptive finite element method with conforming Lagrange elements for the Poisson problem for D\"orfler marking with any bulk parameter $\theta\in(0,1)$. This removes the smallness restriction on the bulk parameter that has persisted in the classical theory of optimal adaptive finite element methods.
\end{abstract}

\thanks{This work was funded by the Deutsche Forschungsgemeinschaft (DFG, German Research Foundation) – SFB 1283/2 2021 – 317210226.}

\maketitle


\section{Introduction}
\label{sec:introduction}

By combining a posteriori error estimates with local mesh refinement, adaptive finite element methods automatically resolve singularities and localized solution features, which would otherwise require prohibitively fine uniform meshes or a priori knowledge of the solution's singularities. However, their convergence analysis is intricate.  Early convergence results include the one-dimensional work of Babuška and Vogelius \cite{BabuskaVogelius1984}, followed in several dimensions by \cite{Doerfler1996,MorinNochettoSiebert2000,MorinNochetteSiebert02}. Mekchay and Nochetto \cite{MekchayNochetto2005} proved contraction of the sum of the energy error and data oscillation, which became an important tool in subsequent AFEM convergence and optimality analyses.  By combining the AFEM loop with a coarsening routine, Binev, Dahmen, and DeVore proved rate-optimality in \cite{BinevDahmenDeVore2004}. Stevenson subsequently established rate-optimality without coarsening in his seminal contribution \cite{Stevenson2007}, under an interior-node property of the refinement.
Building on this line of analysis, Diening and Kreuzer \cite{DieningKreuzer08} proved contraction for the $p$-Laplace problem, including the linear case $p=2$, and thereby, for the first time, dispensed entirely with separate marking for data oscillation.  Cascón, Kreuzer, Nochetto, and Siebert \cite{CasconKreuzerNochettoSiebert2008} further developed these ideas for general symmetric second-order elliptic problems and established rate-optimality with minimal refinement, thereby removing the interior-node property. These results were subsequently extended and generalized in various directions, cf.~\cite{CasconNochetto2012,KreuzerSiebert2011,BelenkiDieningKreuzer2012,BonitoNochetto2010,CarstensenPeterseimRabus2013,FeischlFuehrerPraetorius2014}. The different ingredients of these rate-optimality analyses were later organized and generalized in \cite{CarstensenFeischlPagePraetorius2014} into the four ``axioms of adaptivity''. This framework served as a foundation for further generalizations, including separate marking strategies \cite{CarstensenRabus2017}, saddle-point problems \cite{Feischl2022}, and optimal complexity results including inexact solves \cite{BringmannFeischlMiraciPraetoriusStreitberger2025,BringmannBrunnerPraetoriusStreitberger2025}. Further references can be found in books and overview articles such as \cite{Verfuerth1996,NochettoSiebertVeeser2009,BonitoCanutoNochetteVeeser24}.

Despite this extensive development, a restriction of the classical theory of rate-optimal adaptive finite element methods has remained: the proofs of rate-optimality require sufficiently small bulk parameters $\theta \ll 1$ in the D\"orfler marking routine, see Definition~\ref{def:DoerflerMarking} below. This restriction is not reflected in practical computations, where moderate values such as $\theta\approx0.5$ are commonly used. The present paper removes this restriction. In particular, we prove for any bulk parameter $\theta\in(0,1)$ rate-optimal convergence of the standard adaptive finite element method with conforming Lagrange elements for the Poisson model problem, which, for given bounded polyhedral Lipschitz domain $\Omega\subset\RR^d$ and $f\in L^2(\Omega)$, seeks the solution to
\begin{align*}
  -\Delta u=f\qquad\text{in }\Omega,\qquad u=0\qquad\text{on }\partial\Omega.
\end{align*}

The abstract counterexample of \cite{DieningKreuzer2021} shows that the restriction on the bulk parameter cannot in general be removed using only the four axioms of adaptivity. We therefore introduce a new axiom, \emph{estimator concentration}, in Section~\ref{subsec:EstimatorConcentration}. It asserts that, apart from a contribution controlled by the estimator on a refinement, all but an arbitrarily small fraction of the estimator can be concentrated on a set whose cardinality is proportional to the refinement complexity. This yields an extended comparison lemma valid for every bulk parameter $\theta<1$, which is the key ingredient for the rate-optimality proof. For the Poisson problem, estimator concentration follows from an $\ell^q$ estimate with some $q<2$ for the changes of the local residual indicators on unrefined elements. The required estimate is obtained from fractional Sobolev stability of the Ritz projection, while a best $n$-term argument converts the $\ell^q$ control into estimator concentration.

\textbf{Structure of the paper.}
Section~\ref{sec:Axioms} introduces the adaptive finite element loop and rate-optimality, recalls the four axioms of adaptivity, and presents the new axiom \emph{estimator concentration} together with the extended comparison lemma and the resulting rate-optimality theorem for arbitrary bulk parameters.  Section~\ref{sec:PMP} verifies estimator concentration for the Poisson problem with conforming Lagrange elements of arbitrary fixed polynomial degree. A central ingredient is the fractional Sobolev stability of the Ritz projection, whose proof uses corresponding stability of the Scott--Zhang interpolation operator. The latter is established in Appendix~\ref{app:SZFractionalStability}.

\textbf{Notation.} 
We frequently use the notation $a \lesssim b$ for values $a,b \in \RR$ if $a \leq C \, b$ with some constant $C < \infty$ that might depend on the shape regularity of the underlying triangulation $\tria$, the polynomial degree $k\in \setN$, or the indices $\sigma,q >0$ of the related spaces $W^{\sigma, q}(\Omega)$, but is otherwise independent of the current triangulation. If we have $a \lesssim b\lesssim a$, we write $a\eqsim b$. 


\section{Axioms of adaptivity}
\label{sec:Axioms}

This section verifies rate-optimal convergence of AFEMs on an abstract level. For this, we first recall the notion of rate-optimal convergence in Section~\ref{subsec:ClassOptimality}, then Section~\ref{subsec:FourAxioms} introduces the four axioms of adaptivity in a simplified version of \cite{CarstensenFeischlPagePraetorius2014} and discusses the need for sufficiently small bulk parameters in the existing proofs. Thereafter, Section~\ref{subsec:EstimatorConcentration} introduces the new axiom \emph{estimator concentration}, which in combination with the other four axioms, yields rate-optimal AFEM convergence for all bulk parameters $\theta < 1$. This new axiom will be verified in Section~\ref{sec:PMP} for the Poisson problem.


\subsection{Rate optimality}\label{subsec:ClassOptimality}

Let $\tria_0$ be an initial triangulation and let $\meshes=\meshes(\tria_0)$ denote the family of all admissible triangulations that can be obtained from $\tria_0$ by successive applications of the refinement routine. For $\tria,\tria_\star\in\meshes$, we write $\tria\leq\tria_\star$ if $\tria_\star$ is a refinement of $\tria$. We assume that the family $\meshes$ is uniformly shape regular. Let $V$ denote the continuous energy space. For every $\tria\in\meshes$, let $V(\tria)$ denote the corresponding finite element space and let $u_\tria\in V(\tria)$ be the discrete solution.
Let $\eta_\tria^2\colon\tria\to[0,\infty)$ be a nonnegative error indicator associated with $u_\tria$. We set for all subsets $\mathcal M\subseteq\tria$ the accumulated indicators
\[
  \eta_\tria^2(\mathcal M)\coloneqq\sum_{T\in\mathcal M}\eta_\tria^2(T),
  \quad
  \eta_\tria(\mathcal{M}) \coloneqq \big(\eta_\tria^2(\mathcal{M}) \big)^{1/2},
  \quad\text{and}\quad
  \eta_\tria\coloneqq \eta_\tria(\tria).
\]
The dependence of the error indicators on the corresponding discrete solution is suppressed in the notation. For the Poisson model problem considered in Section~\ref{sec:PMP}, $\eta_\tria$ is the standard residual error estimator.
To formulate rate-optimality, we distinguish the error estimator $\eta_\tria$ from an abstract error quantity, which we denote by $\err(\tria)$. For the Poisson problem considered in Section~\ref{sec:PMP}, $\err(\tria)$ is the total error, i.e., the combination of the gradient error and data oscillation, cf.~\cite{MekchayNochetto2005,CasconKreuzerNochettoSiebert2008}.
Starting with an initial triangulation $\tria_0$, we apply the AFEM loop
\begin{align*}
  \text{\texttt{SOLVE} -- \texttt{ESTIMATE} -- \texttt{MARK} -- \texttt{REFINE}}.
\end{align*}
More precisely, the steps are as follows:
\begin{tabbing}
  \quad \= \texttt{SOLVE} \qquad \= Compute the discrete solution $u_{\tria_\ell} \in V(\tria_\ell)$.
  \\[2mm]
  \> \texttt{ESTIMATE} \qquad \> Calculate the error
  estimators~$\eta_{\tria_\ell}(T)$ for all $T\in \tria_\ell$. 
  \\[2mm]
  \> \texttt{MARK} \qquad \> Mark a set~$\mathcal{M}_\ell \subset \tria_\ell$ of simplices for refinement. 
  \\[2mm]
  \>\texttt{REFINE} \qquad \> Refine the marked simplices to
  get~$\tria_{\ell+1} \coloneqq \mathtt{refine(\tria_\ell,\mathcal{M}_\ell)}$.
\end{tabbing}

\smallskip\noindent%
The loop then repeats with $\ell\mapsto\ell+1$. For simplicity, we assume that the \texttt{SOLVE} step computes the exact discrete solution. The \texttt{REFINE} step leads to a refinement $\tria_{\ell+1}$ of $\tria_\ell$.
We require the following two properties of the refinement routine.

\medskip
\noindent\textup{\textbf{Closure estimate.}}
There exists a constant $\CBDD<\infty$ such that we have, for all sets of marked elements $\mathcal M_\ell\subset\tria_\ell$ with $\tria_{\ell+1}=\mathtt{refine}(\tria_\ell,\mathcal M_\ell)$ for $\ell=0,\dots,L-1$,
\begin{align}\label{eq:Closure}
  \#\tria_L-\#\tria_0\leq\CBDD\sum_{\ell=0}^{L-1}\#\mathcal M_\ell
  \qquad\text{for all }L\in\setN.
\end{align}

\noindent\textup{\textbf{Overlay estimate.}}
Any two triangulations $\tria_1,\tria_2\in\meshes$ admit a common refinement $\tria_1\oplus\tria_2\in\meshes$ such that
\begin{align}\label{eq:Overlay}
  \#(\tria_1\oplus\tria_2)\leq\#\tria_1+\#\tria_2-\#\tria_0.
\end{align}

As shown by Binev, Dahmen, and DeVore in \cite{BinevDahmenDeVore2004} the newest vertex bisection satisfies these properties, where $\tria_1 \oplus \tria_2$ is the smallest joint refinement of~$\tria_1$ and $\tria_2$.
Stevenson generalized in~\cite{Stevenson2008} this statement to higher dimensions using the bisection routines of Kossaczk{\'y}~\cite{Kossaczky1994}, Maubach~\cite{Maubach1995}, and Traxler~\cite{Traxler1997} under suitable assumptions on $\tria_0$. These restrictions on $\tria_0$ have been removed in~\cite{DieningGehringStorn2026}.
The remaining step is \texttt{MARK}, which uses the following marking strategy.

\begin{definition}[D\"orfler marking]
  \label{def:DoerflerMarking}
  Given a bulk parameter $\theta\in(0,1)$, we select a set $\mathcal M_j\subset\tria_j$ of almost minimal cardinality satisfying
  \begin{align}
    \label{eq:marking-criterion}
    \theta\eta_{\tria_j}^2
    \leq \eta_{\tria_j}^2(\mathcal M_j).
  \end{align}
Almost minimal cardinality means that the cardinality of the chosen set is bounded by a fixed constant $C_{\rm mark}\ge1$ times the minimal possible cardinality. 
\end{definition}
Our main target is the verification of rate-optimal convergence of the AFEM loop for any bulk parameter $\theta\in(0,1)$. 
To define rate-optimality, we denote the set of  admissible triangulations with at most $N\in\setN_0$ additional elements by
\begin{align*}
  \meshes_N \coloneqq \lbrace \tria\in\meshes\colon \#\tria-\#\tria_0\leq N\rbrace.
\end{align*}
\begin{definition}[Approximation classes]\label{def:ApproxClass}
  For any $s>0$ and solution $u\in V$ with given right-hand side $f\in V^*$, we define
  \begin{align*}
    \lVert (u,f) \rVert_{\bbA_s^{\err}}
    &\coloneqq \sup_{N\in\setN_0}(1+N)^s\min_{\tria\in\meshes_N}\err(\tria),\\
    \lVert (u,f) \rVert_{\bbA_s^\eta}
    &\coloneqq \sup_{N\in\setN_0}(1+N)^s\min_{\tria\in\meshes_N}\eta_\tria.
  \end{align*}
  The corresponding nonlinear approximation classes are
  \begin{align*}
    \bbA_s^{\err}
    &\coloneqq \bigl\{(u,f)\in V\times V^*\colon \lVert(u,f)\rVert_{\bbA_s^{\err}}<\infty\bigr\},\\
    \bbA_s^\eta
    &\coloneqq \bigl\{(u,f)\in V\times V^*\colon \lVert(u,f)\rVert_{\bbA_s^\eta}<\infty\bigr\}.
  \end{align*}
\end{definition}
The class $\bbA_s^{\err}$ measures the best possible decay rate of the error and will be our primary notion of approximation class. The estimator-based class $\bbA_s^\eta$ corresponds to the formulation used in the axiomatic framework of \cite{CarstensenFeischlPagePraetorius2014}. As discussed in Section~\ref{subsec:FourAxioms}, reliability and efficiency imply $\err(\tria)\eqsim\eta_\tria$ and therefore $\bbA_s^{\err}=\bbA_s^\eta$ with equivalent quasi-norms.
%
\begin{definition}[Rate-optimality]\label{def:ClassOptimal}
AFEM is rate-optimal with respect to the error for a class of right-hand sides $f\in\Sigma\subset V^*$ if for every rate $s>0$ there exists a constant $C(s)<\infty$ such that for any solution $u\in V$ with right-hand side $f\in\Sigma$
  \begin{align*}
    \sup_{j\in\mathbb N_0}(\#\tria_j-\#\tria_0+1)^s\err(\tria_j)
    \leq C(s)\,\lVert(u,f)\rVert_{\bbA_s^{\err}}.
  \end{align*}
  Analogously,  AFEM is rate-optimal with respect to the error estimator if
  \begin{align*}
    \sup_{j\in\mathbb N_0}
    (\#\tria_j-\#\tria_0+1)^s\eta_{\tria_j}
    \leq C(s)\,\lVert(u,f)\rVert_{\bbA_s^\eta}.
  \end{align*}
\end{definition}


\subsection{The four axioms}
\label{subsec:FourAxioms}

The analysis of adaptive finite element methods is commonly formulated in terms of the four ``axioms of adaptivity'' of \cite{CarstensenFeischlPagePraetorius2014}. We use the following simplified version, cf.~\cite{DieningKreuzer2021}. Besides the estimator $\eta_\tria$, let $\delta(\tria,\tria_\star)\geq0$ denote a discrete error difference associated with two nested meshes $\tria,\tria_\star\in\meshes$ with $\tria\leq\tria_\star$.

\begin{definition}[Axioms of adaptivity -- simplified]
  \leavevmode%
  \begin{enumerate}[label=\textup{(\textbf{A\arabic*})},ref=\textup{(A\arabic*)}]
  \item \label{axiom:stability}
    \textup{\textbf{Stability.}} There exists a constant $C_1<\infty$ such that, for all $\tria,\tria_\star\in\meshes$ with $\tria\leq\tria_\star$,
    \begin{align*}
      \bigl|\eta_\tria(\tria\cap\tria_\star)-\eta_{\tria_\star}(\tria\cap\tria_\star)\bigr|
      \leq C_1\delta(\tria,\tria_\star).
    \end{align*}

  \item \label{axiom:reduction}
    \textup{\textbf{Reduction.}} There exist constants $\rho\in[0,1)$ and $C_2<\infty$ such that, for all $\tria,\tria_\star\in\meshes$ with $\tria\leq\tria_\star$,
    \begin{align*}
      \eta_{\tria_\star}^2(\tria_\star\setminus\tria)
      \leq \rho\,\eta_\tria^2(\tria\setminus\tria_\star)
      +C_2\delta(\tria,\tria_\star)^2.
    \end{align*}

  \item \label{axiom:quasi-orthogonality}
    \textup{\textbf{Quasi-orthogonality.}} There exists a constant $C_3<\infty$ such that, for every nested sequence $\{\tria_k\}_{k\in\mathbb N_0}\subset\meshes$,
    \begin{align*}
      \sum_{k=K}^{\infty}\delta(\tria_k,\tria_{k+1})^2
      \leq C_3\eta_{\tria_K}^2
      \qquad\text{for all }K\in\mathbb N_0.
    \end{align*}

  \item \label{axiom:discrete-reliability}
    \textup{\textbf{Discrete reliability.}} There exists a constant $C_4<\infty$ such that, for all $\tria,\tria_\star\in\meshes$ with $\tria\leq\tria_\star$,
    \begin{align*}
      \delta(\tria,\tria_\star)^2
      \leq C_4\eta_\tria^2(\tria\setminus\tria_\star).
    \end{align*}
  \end{enumerate}
\end{definition}

The four axioms are formulated solely in terms of the estimator $\eta_\tria$ and the discrete error difference~$\delta$. To relate the estimator-based approximation class $\bbA_s^\eta$ to the error-based class $\bbA_s^{\err}$, we use the following two standard properties.

\smallskip
\noindent\textup{\textbf{Reliability.}}
There exists a constant $C_{\mathrm{rel}}<\infty$ such that
\begin{align}\label{eq:Reliability}
  \err(\tria)^2\leq C_{\mathrm{rel}}\,\eta_\tria^2
  \qquad\text{for all }\tria\in\meshes.
\end{align}

\noindent\textup{\textbf{Efficiency.}}
There exists a constant $C_{\mathrm{eff}}>0$ such that
\begin{align}\label{eq:Efficiency}
  C_{\mathrm{eff}}\,\eta_\tria^2\leq\err(\tria)^2
  \qquad\text{for all }\tria\in\meshes.
\end{align}
In the framework of \cite{CarstensenFeischlPagePraetorius2014}, reliability follows from discrete reliability together with the mild requirement that the exact solution can be approximated arbitrarily well by admissible refinements; cf.~\cite[Lem.~3.4]{CarstensenFeischlPagePraetorius2014}. Here, we state reliability separately since $\err(\tria)$ is an independent abstract error quantity. Efficiency is an additional property. If both estimates hold, then $\err(\tria)\eqsim\eta_\tria$ and consequently $\bbA_s^{\err}=\bbA_s^\eta$ with equivalent quasi-norms.
The four axioms imply rate-optimality with respect to the estimator for sufficiently small bulk parameters. More precisely, \cite[Thm.~4.1(ii)]{CarstensenFeischlPagePraetorius2014} yields with our simplified  (or their more general) axioms \ref{axiom:stability}\,--\,\ref{axiom:discrete-reliability} the following.

\begin{theorem}[Rate-optimality]\label{thm:ClassOptimality}
  Suppose that \ref{axiom:stability}\,--\,\ref{axiom:discrete-reliability} hold and that the refinement routine satisfies the closure and overlay estimates. Then AFEM with D\"orfler marking is rate-optimal with respect to the error estimator for all
  \begin{align*}
    0<\theta<\theta_\star\coloneqq(1+C_1^2C_4)^{-1}.
  \end{align*}
  If, in addition, reliability and efficiency \eqref{eq:Reliability}--\eqref{eq:Efficiency} hold, then AFEM is also rate-optimal with respect to the error.
\end{theorem}

The severe restriction $\theta<\theta_\star$ enters through the following comparison lemma. The underlying comparison principle goes back to Stevenson \cite{Stevenson2007}; its formulation in the axiomatic setting is given in \cite[Prop.~4.12]{CarstensenFeischlPagePraetorius2014}.

\begin{lemma}[Comparison lemma]\label{lem:Comparison}
  Suppose that \ref{axiom:stability} and \ref{axiom:discrete-reliability} hold. For every $0<\theta<\theta_\star$, there exists a $\kappa\in(0,1)$ such that, for all $\tria,\tria_\star\in\meshes$ with $\tria\leq\tria_\star$,
  \begin{align*}
    \eta_{\tria_\star}^2\leq\kappa\,\eta_\tria^2
    \qquad\text{implies}\qquad
    \theta\eta_\tria^2\leq\eta_\tria^2(\tria\setminus\tria_\star).
  \end{align*}
\end{lemma}

The lemma shows that, if a refinement $\tria_\star\geq\tria_\ell$ sufficiently reduces the error estimator, then the refined elements $\tria_\ell\setminus\tria_\star$ satisfy the D\"orfler criterion. Consequently, the almost minimality of the marked set $\mathcal M_\ell$ implies
\begin{align}\label{eq:ConsCompLem}
  \#\mathcal M_\ell
  \leq C_{\rm mark}
  \#(\tria_\ell\setminus\tria_\star)
  \leq C_{\rm mark}
  \big(\#\tria_\star-\#\tria_\ell\big).
\end{align}
This estimate allows us to compare the number of marked elements with the complexity of a suitable refinement constructed from the approximation class. However, the comparison lemma requires $\theta<\theta_\star$. In fact, the abstract counterexample of \cite{DieningKreuzer2021} shows that a restriction on the bulk parameter~$\theta$ cannot in general be removed using only the four axioms~\ref{axiom:stability}--\,\ref{axiom:discrete-reliability}. This motivates introducing a novel axiom in the following subsection.


\subsection{Estimator concentration}
\label{subsec:EstimatorConcentration}

The classical comparison lem\-ma (Lemma~\ref{lem:Comparison}) states that a sufficiently strong reduction of the estimator on a refinement $\tria_\star\geq\tria$ implies that the set of refined elements $\tria\setminus\tria_\star$ satisfies the D\"orfler marking criterion~\eqref{eq:marking-criterion}, provided $\theta<\theta_\star$. To prove rate-optimality for every $\theta<1$, we need such a comparison for arbitrary bulk parameters. For the cardinality argument, however, the D\"orfler criterion need not be satisfied by the refined elements themselves. It suffices to find a set $\mathcal S\subseteq\tria$ satisfying the D\"orfler criterion whose cardinality is controlled by the number of elements added by the refinement. We therefore aim for the following extension of Lemma~\ref{lem:Comparison}. We will show that estimator concentration, introduced below, provides exactly this extension.

\begin{lemma}[Extended comparison lemma]
  \label{lem:ExtendedComparison}
  Suppose that estimator concentration, as formulated in Axiom~\ref{axiom:EstimatorConcentration} below, holds.  For every $\theta\in(0,1)$, there exist constants $\kappa_\theta\in(0,1)$ and $C_\theta<\infty$ such that
  \begin{align*}
    \eta_{\tria_\star}^2\leq\kappa_\theta\,\eta_\tria^2\qquad\text{with }\tria,\tria_\star\in\meshes \text{ and } \tria\leq\tria_\star
  \end{align*}
  implies the existence of a subset $\mathcal S\subseteq\tria$ satisfying
  \begin{align*}
    \theta\eta_\tria^2
    \leq\eta_\tria^2(\mathcal S)
    \qquad\text{and}\qquad
    \#\mathcal S
    \leq C_\theta\bigl(\#\tria_\star-\#\tria\bigr).
  \end{align*}
\end{lemma}

Compared with the classical comparison lemma, the restriction $\theta<\theta_\star$ is removed at the expense of replacing the set of refined elements $\tria\setminus\tria_\star$ by a set $\mathcal S$ whose cardinality is controlled by the refinement complexity.
The following new axiom guarantees precisely this property.

\begin{axiom}[Estimator concentration]
  \label{axiom:EstimatorConcentration}
  For every $\epsilon>0$, there exist constants $A_\epsilon<\infty$ and $B_\epsilon\geq1$ such that, for all $\tria,\tria_\star\in\meshes$ with $\tria\leq\tria_\star$, there exists a subset $\mathcal S\subseteq\tria$ satisfying
  \begin{align*}
    \#\mathcal S
    &\leq A_\epsilon\bigl(\#\tria_\star-\#\tria\bigr)
    \qquad\text{and}\qquad
    \eta_\tria^2(\tria\setminus\mathcal S)
    \leq\epsilon\,\eta_\tria^2+B_\epsilon\,\eta_{\tria_\star}^2.
  \end{align*}
\end{axiom}

Estimator concentration states that, apart from a contribution controlled by the estimator on the refined mesh, all but an arbitrarily small fraction of the estimator on $\tria$ can be concentrated on a set whose cardinality is proportional to the refinement complexity.

\begin{proof}[Proof of Lemma~\ref{lem:ExtendedComparison}]
  Fix $\theta\in(0,1)$ and set $\epsilon_\theta\coloneqq(1-\theta)/2$.  Let $A_{\epsilon_\theta}$ and $B_{\epsilon_\theta}$ denote the constants provided by estimator concentration, cf.~Axiom~\ref{axiom:EstimatorConcentration}, and set
  \begin{align*}
    \kappa_\theta\coloneqq\frac{1-\theta}{2B_{\epsilon_\theta}}\in(0,1).
  \end{align*}
  If $\eta_{\tria_\star}^2\leq\kappa_\theta\eta_\tria^2$, then estimator concentration provides a subset $\mathcal S\subseteq\tria$ with
  \begin{align*}
    \eta_\tria^2(\tria\setminus\mathcal S)
    &\leq\epsilon_\theta\eta_\tria^2+B_{\epsilon_\theta}\eta_{\tria_\star}^2 \leq (\epsilon_\theta + \kappa_\theta B_{\epsilon_\theta}) \eta_\tria^2
    \leq(1-\theta)\eta_\tria^2.
  \end{align*}
  Consequently, we obtain
  \begin{align*}
    \theta\eta_\tria^2
    \leq\eta_\tria^2-\eta_\tria^2(\tria\setminus\mathcal S)
    =\eta_\tria^2(\mathcal S).
  \end{align*}
  The cardinality estimate follows with $C_\theta\coloneqq A_{\epsilon_\theta}$.
\end{proof}

Suppose that a refinement $\tria_\star\geq\tria_\ell$ satisfies the estimator reduction in Lemma~\ref{lem:ExtendedComparison}. The resulting set $\mathcal S\subseteq\tria_\ell$ satisfies the D\"orfler marking criterion, and the almost minimality of the marked set $\mathcal M_\ell$ therefore yields
\begin{align}\label{eq:ConsCompLem2}
  \#\mathcal M_\ell
  \leq C_{\rm mark}\#\mathcal S
  \leq C_{\rm mark}C_\theta
  \bigl(\#\tria_\star-\#\tria_\ell\bigr).
\end{align}
The additional constant is harmless in the standard cardinality argument; such a fixed multiplicative factor is already allowed in the proof of \cite[Lem.~5.2]{Stevenson2007}. In contrast to the classical comparison lemma, the extended comparison lemma is available for every $\theta\in(0,1)$.

\begin{theorem}[Rate-optimality for arbitrary bulk parameters]
  \label{thm:ClassOptimalityAllTheta}
  Suppose that \ref{axiom:stability}\,--\,\ref{axiom:discrete-reliability} hold and that estimator concentration, as formulated in Axiom~\ref{axiom:EstimatorConcentration}, holds. Moreover, let the refinement routine satisfy the closure and overlay estimates. Then, for any fixed bulk parameter $\theta\in(0,1)$, AFEM with D\"orfler marking is rate-optimal with respect to the error estimator.  If, in addition, reliability and efficiency \eqref{eq:Reliability}--\eqref{eq:Efficiency} hold, then AFEM is also rate-optimal with respect to the error.
\end{theorem}

\begin{proof}
  The first assertion follows from the standard cardinality proof of \cite{Stevenson2007,CarstensenFeischlPagePraetorius2014}, with Lemma~\ref{lem:ExtendedComparison} replacing the classical comparison lemma. More precisely, the approximation class $\bbA_s^\eta$, the overlay estimate, and the standard quasi-monotonicity consequence of the axioms provide, for every adaptive triangulation $\tria_\ell$ with $\eta_{\tria_\ell}>0$, a refinement $\tria_\star\geq\tria_\ell$ such that
  \begin{align*}
    \eta_{\tria_\star}^2
    &\leq\kappa_\theta\,\eta_{\tria_\ell}^2\qquad\text{and}\qquad
    \#\tria_\star-\#\tria_\ell
    \lesssim
    \left(
      \frac{\lVert(u,f)\rVert_{\bbA_s^\eta}}
           {\eta_{\tria_\ell}}
    \right)^{1/s}.
  \end{align*}
  Lemma~\ref{lem:ExtendedComparison}, almost minimal D\"orfler marking, and the closure estimate yield the required cardinality bound. Combining this bound with linear convergence and the usual summation argument, cf.~\cite[Proof of Prop.~4.15]{CarstensenFeischlPagePraetorius2014}, proves rate-optimality with respect to the error estimator.
  If additionally reliability and efficiency hold, then $\eta_\tria\eqsim\err(\tria)$ and $\bbA_s^\eta=\bbA_s^{\err}$ with equivalent quasi-norms. Hence rate-optimality with respect to the estimator implies rate-optimality with respect to the error.
\end{proof}


\section{Estimator concentration for the Poisson model problem}
\label{sec:PMP}

In this section, we verify estimator concentration for the Poisson model problem. We first introduce the residual estimator and a total energy formulation of the error and then state the main result of this section.


\subsection{Poisson problem and total energy}
\label{subsec:PoissonSetup}

Let $\Omega\subset\RR^d$, $d\geq2$, be a bounded polyhedral Lipschitz domain and let $f\in L^2(\Omega)$. We specialize to the family $\meshes=\meshes(\tria_0)$ generated by newest-vertex bisection. In two dimensions, the initial triangulation may be equipped with matching refinement edges as in~\cite{BinevDahmenDeVore2004}; alternatively, and in arbitrary dimension, one may use the coloring-based initialization of~\cite{DieningGehringStorn2026}. We seek the weak solution $u\in H^1_0(\Omega)$ to the Poisson model problem
\begin{align}\label{eq:PMP}
  \int_\Omega\nabla u\cdot\nabla v\dx
  =\int_\Omega fv\dx
  \qquad\text{for all }v\in H^1_0(\Omega).
\end{align}
Equivalently, $u$ is the minimizer over $H^1_0(\Omega)$ of the Dirichlet energy
\begin{align}\label{eq:DirichletEnergy}
  \mathcal{J}(v)
  \coloneqq
  \tfrac12\|\nabla v\|_{L^2(\Omega)}^2-\int_\Omega fv\dx.
\end{align}
For a fixed polynomial degree $k\in\setN$ and $\tria\in\meshes$, we denote for all $T\in \tria$ by $\mathbb{P}_k(T)$ the space of polynomials on $T$ of maximal degree $k$ and define the discrete space
\begin{align*}
  V(\tria)
  \coloneqq
  \bigl\{
    v\in H^1_0(\Omega)\colon
    v|_T\in\mathbb P_k(T)\text{ for all }T\in\tria
  \bigr\}.
\end{align*}
Let $u_\tria\in V(\tria)$ denote the Galerkin solution to \eqref{eq:PMP}, equivalently the minimizer of $\mathcal{J}$ over $V(\tria)$.
We denote by $\mathcal F_\tria^{\mathrm{int}}$ the set of interior facets of $\tria$. For every $F\in\mathcal F_\tria^{\mathrm{int}}$, let $n_F$ be a fixed unit normal vector and $T^+,T^-\in\tria$ the adjacent simplices,
with $n_F$ pointing from $T^+$ to $T^-$. We define for $v\in V(\tria)$ the jump across $F\in\mathcal F_\tria^{\mathrm{int}}$ by
\begin{align*}
\llbracket \nabla v\cdot n_F\rrbracket \coloneqq (\nabla v|_{T^+}-\nabla v|_{T^-})\cdot n_F.
\end{align*}
The residual error indicator, cf.~\cite{Verfuerth1996,AinsworthOden2000}, is defined with the elementwise application of the Laplace operator $\Delta_h$ and with $h_T\coloneqq|T|^{1/d}$ for $T\in\tria$ by
\begin{align}\label{eq:PoissonEstimator}
  \eta_\tria(T)^2
  \coloneqq
  h_T^2\|f+\Delta_hu_\tria\|_{L^2(T)}^2
  +
  h_T\sum_{{F\subset\partial T \cap \mathcal F_\tria^{\mathrm{int}}}}
  \|\llbracket\nabla u_\tria\cdot n_F\rrbracket\|_{L^2(F)}^2.
\end{align}
Let $\Pi_T^{k-1}\colon L^2(T)\to\mathbb P_{k-1}(T)$ denote the $L^2(T)$-orthogonal projection. We define the data oscillation by
\begin{align}\label{eq:PoissonOscillation}
  \osc_\tria^2
  \coloneqq
  \sum_{T\in\tria}
  h_T^2\bigl\|f-\Pi_T^{k-1}f\bigr\|_{L^2(T)}^2.
\end{align}
\begin{remark}[Choice of oscillation]
  \label{rem:choice-of-oscillation}
  The polynomial degree in the definition of the data oscillation is not unique. More generally, one may replace $k-1$ in \eqref{eq:PoissonOscillation} by any fixed $r\geq\max\{0,k-2\}$. Since $\Delta u_\tria|_T\in\mathbb P_{k-2}(T)$, the corresponding oscillation is controlled by the residual estimator, while the standard efficiency estimate holds with a constant that may depend on $r$. In particular, higher-order choices such as $r=2k-2$ are also possible, cf.~\cite{CasconKreuzerNochettoSiebert2008}. We use $r=k-1$ throughout since it gives a simple formulation for all $k\geq1$ and agrees with the oscillation used in the verification of the axioms in \cite[Prop.~5.1]{CarstensenFeischlPagePraetorius2014}.
\end{remark}

For $\tria\in\meshes$, we define the \emph{total energy} $\mathcal{G}(\tria)\coloneqq \mathcal{J}(u_\tria)+\osc_\tria^2$ and $\mathcal{G}_\infty \coloneqq  \mathcal{J}(u)$.
The error quantity introduced in Section~\ref{subsec:ClassOptimality} is the corresponding total error
\begin{align}\label{eq:PoissonTotalError}
  \err(\tria)^2
  \coloneqq
  \mathcal{G}(\tria)-\mathcal{G}_\infty
  =
  \tfrac12\|\nabla(u-u_\tria)\|_{L^2(\Omega)}^2+\osc_\tria^2.
\end{align}
The last identity follows from Galerkin orthogonality. Thus, up to the immaterial normalization of the gradient error, $\err(\tria)$ is the usual total error consisting of the gradient error and the data oscillation, cf.~\cite{MekchayNochetto2005,CasconKreuzerNochettoSiebert2008}.
For nested triangulations $\tria\leq\tria_\star$, we define the discrete error difference by
\begin{align}\label{eq:PoissonDiscreteDifference}
  \delta(\tria,\tria_\star)^2
  \coloneqq
  \mathcal{G}(\tria)-\mathcal{G}(\tria_\star)
  =
  \tfrac12\|\nabla(u_{\tria_\star}-u_\tria)\|_{L^2(\Omega)}^2
  +\osc_\tria^2-\osc_{\tria_\star}^2.
\end{align}
The oscillation is monotone under refinement, and hence $\delta(\tria,\tria_\star)\geq0$. Moreover,
\begin{align}\label{eq:TotalErrorDifference}
  \delta(\tria,\tria_\star)^2
  =
  \err(\tria)^2-\err(\tria_\star)^2.
\end{align}

\begin{remark}[Relation to the maximum strategy]
  The total energy $\mathcal{G}(\tria)=\mathcal{J}(u_\tria)+\osc^2_\tria$ incorporates the data oscillation into the Dirichlet energy, so that both the total error and the discrete error difference are exact differences of $\mathcal{G}$. This energy viewpoint is closely related to the instance-optimality analysis of the adaptive maximum strategy in \cite{DieningKreuzerStevenson2016}; there, the squared oscillation is eventually replaced by a related data quantity $\mathcal H$ in the definition of the total energy $\mathcal{J}+\mathcal H$.
\end{remark}

Our choice of $\delta$ differs slightly from the discrete distance used in \cite{CarstensenFeischlPagePraetorius2014}, which for the Poisson problem consists only of the discrete gradient difference $\|\nabla(u_{\tria_\star}-u_\tria)\|_{L^2(\Omega)}$. The definition \eqref{eq:PoissonDiscreteDifference} is more natural for the present total error formulation: it identifies $\delta^2$ exactly with the decrease of the total energy and makes quasi-orthogonality an immediate telescoping property. The standard verification of the axioms requires only minor modifications. Stability and reduction follow from their corresponding estimates for the discrete gradient difference since
\begin{align*}
  \|\nabla(u_{\tria_\star}-u_\tria)\|_{L^2(\Omega)}
  \leq\sqrt2\,\delta(\tria,\tria_\star).
\end{align*}
Furthermore, \eqref{eq:TotalErrorDifference} and reliability yield quasi-orthogonality by telescoping. Discrete reliability for $\delta$ follows by the usual discrete reliability estimate together with
\begin{align*}
  0
  \leq\osc_\tria^2-\osc_{\tria_\star}^2
  \leq\sum_{T\in\tria\setminus\tria_\star}
       h_T^2\|f-\Pi_T^{k-1}f\|_{L^2(T)}^2
  \leq\eta_\tria^2(\tria\setminus\tria_\star)
\end{align*}
Thus the axioms \ref{axiom:stability}\,--\,\ref{axiom:discrete-reliability}, as well as reliability and efficiency \eqref{eq:Reliability}--\eqref{eq:Efficiency}, hold for the quantities defined above; cf.~\cite[Prop.~5.1]{CarstensenFeischlPagePraetorius2014}. In particular, we have
\begin{align*}
  \err(\tria)\eqsim\eta_\tria
  \qquad\text{for all }\tria\in\meshes.
\end{align*}
In view of Theorem~\ref{thm:ClassOptimalityAllTheta}, it remains to verify estimator concentration.

\begin{theorem}[Estimator concentration for the Poisson problem]
  \label{thm:PoissonEstimatorConcentration}
  Let $k\in\setN$ be fixed. Then the residual estimator \eqref{eq:PoissonEstimator} satisfies estimator concentration in the sense of Axiom~\ref{axiom:EstimatorConcentration}.  Consequently, for every bulk parameter $\theta\in(0,1)$, AFEM is rate-optimal with respect to the error estimator and the total error \eqref{eq:PoissonTotalError}.
\end{theorem}

The proof of Theorem~\ref{thm:PoissonEstimatorConcentration} occupies the remainder of this section. We first show that estimator concentration follows from an $\ell^q$ estimate, for some $q<2$, for the changes of the local error indicators on unrefined elements. The main analytic ingredient for this estimate is fractional Sobolev stability of the Ritz projection. To exploit this stability, we represent the difference $u_{\tria_\star}-u_\tria$ through a finite element function supported on only $\mathcal O(\#\tria_\star-\#\tria)$ elements. A best $n$-term argument then converts the resulting $\ell^q$ estimate into estimator concentration.


\subsection{\texorpdfstring{An $\ell^q$-criterion for estimator concentration}{An lq-criterion for estimator concentration}}\label{subsec:EstimatorChanges}

We first reduce estimator concentration to an $\ell^q$ estimate for the changes of the local error indicators. The essential point is that the exponent $q$ is strictly smaller than two.
Note that this approach does not use any specific properties of our discretization and thus applies to any problem fitting into the abstract framework of the axioms in Section~\ref{subsec:FourAxioms}.

\begin{lemma}[Estimator-change criterion]
  \label{lem:EstimatorChangeCriterion}
  Assume~\ref{axiom:discrete-reliability} is true and suppose that there exist $q\in(1,2)$ and $C_{\rm ch}<\infty$ such that, for all $\tria,\tria_\star\in\meshes$ with $\tria\leq\tria_\star$,
  \begin{align}\label{eq:EstimatorChangeCriterion}
    \Big(
      \sum_{T\in\tria\cap\tria_\star}
      \left|\eta_\tria(T)-\eta_{\tria_\star}(T)\right|^q
    \Big)^{  1/q}
    \leq
    C_{\rm ch}
    \bigl(\#\tria_\star-\#\tria\bigr)^{\frac 1q- \frac 12}
    \delta(\tria,\tria_\star).
  \end{align}
  Then estimator concentration, as formulated in Axiom~\ref{axiom:EstimatorConcentration}, holds.
\end{lemma}

The proof is based on the following best $n$-term estimate, cf.~\cite[Sec.~2]{DeVore1998}.

\begin{lemma}[Best $n$-term approximation]
  \label{lem:BestNterm}
  Let $q\in(0,2)$, let $(a_i)_{i\in I}\subset\RR_{\geq0}$ be a finite sequence, and let $n\in\mathbb N$. If $\#I\leq n$, set $I_n\coloneqq I$. Otherwise, let $I_n\subseteq I$ contain the indices of the $n$ largest values $a_i$. Then
  \begin{align*}
    \Big(
      \sum_{i\in I\setminus I_n}a_i^2
    \Big)^{1/2}
    \leq
    n^{\frac 12- \frac 1q}
    \Big(
      \sum_{i\in I}a_i^q
    \Big)^{1/q}.
  \end{align*}
\end{lemma}

\begin{proof}
  If $\#I\leq n$, there is nothing to prove. Otherwise, let $a_1^\star\geq a_2^\star\geq\cdots$ denote the decreasing rearrangement of $(a_i)_{i\in I}$. Then
  \begin{align*}
    n(a_{n+1}^\star)^q
    \leq
    \sum_{i\in I}a_i^q,
  \end{align*}
  and therefore
  \begin{align*}
    \sum_{i\in I\setminus I_n}a_i^2
    =
    \sum_{j>n}(a_j^\star)^2
    \leq
    (a_{n+1}^\star)^{2-q}
    \sum_{j>n}(a_j^\star)^q
    \leq
    n^{1- \frac 2q}
    \Big(\sum_{i\in I}a_i^q\Big)^{2/q}.
  \end{align*}
  Taking square roots proves the assertion.
\end{proof}

\begin{proof}[Proof of Lemma~\ref{lem:EstimatorChangeCriterion}]
  Fix $\epsilon>0$ and let $\tria\leq\tria_\star$ with $\tria,\tria_\star \in \meshes$. Set
  \begin{align*}
    m\coloneqq\#\tria_\star-\#\tria
    \qquad
\text{and}\qquad    \alpha\coloneqq 1/q- 1/2>0.
  \end{align*}
  
  \textit{Step 1 (The set $\mathcal{S}$).}
  Since the case $m=0$ is trivial, we assume $m\geq1$. We set 
  \begin{align*}
R_\epsilon \coloneqq \max\lbrace 1,     (C_{\rm ch} \sqrt{C_4} \sqrt{2/\epsilon})^{1/\alpha} \rbrace \qquad\text{and}\qquad n \coloneqq \lceil R_\epsilon m \rceil. 
  \end{align*}
 Let $\mathcal B_n=\tria\cap\tria_\star$ if $\#(\tria\cap\tria_\star)\leq n$. Otherwise, let $\mathcal B_n\subseteq\tria\cap\tria_\star$ consist of the elements corresponding to the $n$ largest values of $\left|\eta_\tria(T)-\eta_{\tria_\star}(T)\right|$.
We define
  \begin{align*}
    \mathcal S
    \coloneqq
    (\tria\setminus\tria_\star)\cup\mathcal B_n\qquad\text{and obtain}\qquad
    \#\mathcal S
    \leq
    m+n
    \leq
    (R_\epsilon+2)m.
    \end{align*}
    
  \textit{Step 2 (Estimator concentration).}
Since $\tria\setminus\mathcal S=(\tria\cap\tria_\star)\setminus\mathcal B_n$, the reverse triangle inequality in $\ell^2$, Lemma~\ref{lem:BestNterm}, \eqref{eq:EstimatorChangeCriterion}, and \ref{axiom:discrete-reliability} give
  \begin{align*}
    \eta_\tria(\tria\setminus\mathcal S)
    &\leq
    \eta_{\tria_\star}(\tria\setminus\mathcal S)
    +
    \Big(
      \sum_{T\in\tria\setminus\mathcal S}
      |\eta_\tria(T)-\eta_{\tria_\star}(T)|^2
    \Big)^{1/2}
    \\
& \leq     \eta_{\tria_\star}(\tria\setminus\mathcal S)
    +
 n^{-\alpha}   
        \Big(
      \sum_{T\in\tria\cap \tria_\star}
      |\eta_\tria(T)-\eta_{\tria_\star}(T)|^q
    \Big)^{1/q} \\
    &\leq
    \eta_{\tria_\star}
    +
    n^{-\alpha}
    C_{\rm ch}m^\alpha
    \delta(\tria,\tria_\star) \leq \eta_{\tria_\star}
    +
    \sqrt{\epsilon/2}\,\eta_\tria.
  \end{align*}
  Consequently, we obtain
  \begin{align*}
    \eta_\tria^2(\tria\setminus\mathcal S)
    \leq
    \epsilon\,\eta_\tria^2
    +
    2\,\eta_{\tria_\star}^2.
  \end{align*}
  This proves estimator concentration with $A_\epsilon\coloneqq R_\epsilon+2$ and $B_\epsilon\coloneqq 2$.
\end{proof}

In view of Lemma~\ref{lem:EstimatorChangeCriterion}, it remains to prove \eqref{eq:EstimatorChangeCriterion} for some $q<2$. The key idea is to replace the possibly globally supported Galerkin difference by a sparse projection error with comparable energy norm and recover the difference by a Ritz reconstruction. We first establish the fractional stability of the Ritz projection needed to transfer fractional estimates through this reconstruction. 


\subsection{Fractional Ritz stability}
\label{sec:fract-ritz-stab}

Let $R_\tria$ denote the Ritz projection onto $V(\tria)$ with $\tria\in\meshes$, defined by
\begin{align}\label{eq:DefRitz}
  \int_\Omega\nabla R_\tria v\cdot\nabla\xi\dx
  =
  \int_\Omega\nabla v\cdot\nabla\xi\dx
  \qquad\text{for all }v\in W^{1,1}(\Omega) \text{ and }\xi\in V(\tria).
\end{align}
It is therefore stable in the $H^1_0(\Omega)$-norm. For the estimator-change estimate \eqref{eq:EstimatorChangeCriterion}, we need stability slightly beyond this energy scale. 
To make this statement precise, we set for $0<\sigma<1$, $1<q<\infty$, $\omega\subseteq\Omega$, and $g\in L^q(\omega)$ the Sobolev--Slobodeckij seminorm
\begin{align}\label{eq:def-Wsq}
  \abs{g}_{W^{\sigma,q}(\omega)}^q
  \coloneqq
  \int_\omega\int_\omega
  \frac{|g(x)-g(y)|^q}{|x-y|^{d+\sigma q}}
  \dy\dx.
\end{align}
The main result of this subsection provides stability of $\nabla R_\tria$ in $W^{\sigma,q}(\Omega)$ for some $q<2$ and the critical fractional differentiability $\sigma= d/q- d/2$.

\begin{remark}[Scaling]
  \label{rem:scaling}
  The choice $\sigma- d/q=- d/2$ is the critical scaling corresponding to the Sobolev embedding $W^{\sigma,q}(\omega)\hookrightarrow L^2(\omega)$. This embedding holds, in particular, if $\omega$ is a bounded $W^{\sigma,q}$-extension domain, and hence for bounded Lipschitz domains; cf.~\cite[Thms.~5.4 and~6.7]{DiNezzaPalatucciValdinoci2012}. Moreover, under dilations the $W^{\sigma,q}(\omega)$ seminorm scales with exponent $d/q-\sigma= d/2$, exactly as the $L^2(\omega)$ norm, which we exploit below.
\end{remark}
Classical $W^{1,p}(\Omega)$ stability of the Ritz projection on quasi-uniform meshes goes back to Rannacher and Scott~\cite{RannacherScott1982}; see also the recent pointwise gradient estimate of~\cite{DieningRolfesSalgado2024}. Related results include fractional $H^{1+s}(\Omega)$ stability under quasi-uniform mesh assumptions~\cite[Lem.~A.3]{GuermondPasciak2008} and $W^{1,p}(\Omega)$ stability for $p$ close to two for piecewise affine elements on shape-regular triangulations of convex polygons~\cite{BadrRuss2014}. The latter is obtained by interpolation of invertibility. We combine this principle with fractional Scott--Zhang stability to obtain the required $W^{1+\sigma,q}(\Omega)$ estimate uniformly for all $\tria\in\meshes$ without  quasi-uniformity assumptions.

\begin{lemma}[Ritz stability]
  \label{lem:PoissonCriticalRitz}
  Let $k\in\mathbb N$. The following stability estimates hold uniformly for all $\tria\in \meshes$.
  \begin{enumerate}[itemsep=0.3em]
  \item\label{itm:Ritz-W1p} There exists an $\epsilon>0$ such that, for all $1<p<\infty$ with $\abs{ 1/p- 1/2}\leq\epsilon$,
    \begin{align*}
      \norm{\nabla R_\tria v}_{L^p(\Omega)}
      \lesssim
      \norm{\nabla v}_{L^p(\Omega)}
      \qquad\text{for all }v\in W^{1,p}_0(\Omega).
    \end{align*}
  \item\label{itm:Ritz-Hs} There exists $\tau\in(0,  1/2)$ such that, for all $0\leq s\leq\tau$,
    \begin{align*}
      \norm{R_\tria v}_{H^{1+s}(\Omega)}
      \lesssim
      \norm{v}_{H^{1+s}(\Omega)}
      \qquad\text{for all }v\in H^{1+s}(\Omega)\cap H^1_0(\Omega).
    \end{align*}
  \item\label{itm:Ritz-Wsq} There exist $q\in(1,2)$ and $\sigma\coloneqq d/q- d/2>0$ with $\sigma q<1$ such that
    \begin{align*}
      \abs{\nabla R_\tria v}_{W^{\sigma,q}(\Omega)}
      \lesssim
      \abs{\nabla v}_{W^{\sigma,q}(\Omega)}
      \qquad\text{for all }v\in W^{1+\sigma,q}(\Omega)\cap W^{1,q}_0(\Omega).
    \end{align*}
  \end{enumerate}
  The constants $\epsilon$ and $\tau$ depend on the domain $\Omega$, the shape regularity of $\meshes$, and the polynomial degree $k$. The hidden constants additionally depend on the exponents.
\end{lemma}
\begin{proof}
  Let $\Pi_\tria$ denote the trace-preserving Scott--Zhang projection from Appendix~\ref{app:SZFractionalStability}. Its standard stability gives, for every $1<\rho<\infty$,
  \begin{align}\label{eq:aux-SZ-W1rho}
    \norm{\nabla\Pi_\tria v}_{L^\rho(\Omega)}
    \lesssim
    \norm{\nabla v}_{L^\rho(\Omega)}
    \qquad\text{for all }v\in W^{1,\rho}(\Omega).
  \end{align}
  Moreover, Lemma~\ref{lem:SZFractionalStability} yields, whenever $1<\rho<\infty$ and $0<s< 1/\rho$,
  \begin{align*}
    \abs{\nabla\Pi_\tria v}_{W^{s,\rho}(\Omega)}
    \lesssim
    \abs{\nabla v}_{W^{s,\rho}(\Omega)}.
  \end{align*}
  We first define a bilinear form on $H^1_0(\Omega)\times H^1_0(\Omega)$ for all $v,z\in H^1_0(\Omega)$ by
  \begin{align}\label{eq:def-Btria}
    (B_\tria v)(z)
    \coloneqq
    \int_\Omega\nabla\Pi_\tria v\cdot\nabla\Pi_\tria z\dx
    +
    \int_\Omega\nabla(v-\Pi_\tria v)\cdot\nabla(z-\Pi_\tria z)\dx.
  \end{align}
  The operator $B_\tria$ is the block diagonalization of the elliptic form with respect to the decomposition induced by $\Pi_\tria$, in the spirit of classical subspace correction methods, cf.~\cite{Xu1992}. The interpolation-of-invertibility argument below is related to the strategy for discrete Meyers estimates in \cite[Sec.~4.2]{BadrRuss2014}. We use Shneiberg's interpolation theorem \cite{Shneiberg1973} in the quantitative form of \cite[Thm.~3.1]{MeinlschmidtRehberg2021}.

  \textit{Step 1 (Invertibility in the energy space).}
  The stability estimate \eqref{eq:aux-SZ-W1rho} for $\rho=2$ shows that $B_\tria\colon H^1_0(\Omega)\to H^{-1}(\Omega)$ is uniformly bounded. Moreover, we have for all $v\in H^1_0(\Omega)$
  \begin{align*}
    \tfrac12\norm{\nabla v}_{L^2(\Omega)}^2
    \leq
    \norm{\nabla\Pi_\tria v}_{L^2(\Omega)}^2
    +
    \norm{\nabla(v-\Pi_\tria v)}_{L^2(\Omega)}^2
    =
    (B_\tria v)(v).
  \end{align*}
  Hence the Lax--Milgram theorem yields uniform invertibility of
  \begin{align}\label{eq:B-energy-inverse}
    B_\tria\colon H^1_0(\Omega)\to H^{-1}(\Omega).
  \end{align}

  \textit{Step 2 (Variation of integrability).}
  Let $1<p<\infty$. By combining \eqref{eq:aux-SZ-W1rho} and H\"older's inequality, we obtain for all $v\in W^{1,p}_0(\Omega)$ and $z\in W^{1,p'}_0(\Omega)$
  \begin{align*}
    \abs{(B_\tria v)(z)}
    \lesssim
    \norm{\nabla v}_{L^p(\Omega)}
    \norm{\nabla z}_{L^{p'}(\Omega)}.
  \end{align*}
  Thus $B_\tria\colon W^{1,p}_0(\Omega)\to W^{-1,p}(\Omega)$ is uniformly bounded for $p$ in every compact subinterval of $(1,\infty)$. Since the spaces $W^{1,p}_0(\Omega)$ and $W^{-1,p}(\Omega)$ form complex interpolation scales and \eqref{eq:B-energy-inverse} is uniformly invertible for $p=2$, the quantitative Shneiberg theorem yields the existence of an $\epsilon>0$, independent of $\tria$, such that
  \begin{align}\label{eq:BT-W1p-invertible}
    B_\tria\colon W^{1,p}_0(\Omega)\to W^{-1,p}(\Omega)
    \quad\text{is uniformly invertible if }
    \bigabs{\tfrac1p-\tfrac12}\leq\epsilon.
  \end{align}

  \textit{Step 3 (Variation of smoothness).}
  We use the Bessel potential spaces $H^{t,p}(\Omega)=F^t_{p,2}(\Omega)$. For $1/p<t<1+1/p$, let $H_D^{t,p}(\Omega)$ denote the subspace with homogeneous Dirichlet trace. In particular, we have $H_D^{1,p}(\Omega)=W^{1,p}_0(\Omega)$. We abbreviate $H_D^t(\Omega)\coloneqq H_D^{t,2}(\Omega)$ and fix $s_0\in(0,1/2)$. Lemma~\ref{lem:SZFractionalStability}, applied with $q=2$ and $t=1\pm s_0$, together with the preservation of homogeneous Dirichlet traces, gives
  \begin{align*}
    \norm{\Pi_\tria v}_{H_D^{1\pm s_0}(\Omega)}
    \lesssim
    \norm{v}_{H_D^{1\pm s_0}(\Omega)}.
  \end{align*}
  Since $s_0<1/2$, the spaces $H^{s_0}_0(\Omega)$ and $H^{s_0}(\Omega)$ coincide with equivalent norms. Consequently, differentiation maps $H^{1-s_0}(\Omega)$ boundedly into $H^{-s_0}(\Omega)=(H^{s_0}(\Omega))^*$. We can therefore extend the bilinear form~\eqref{eq:def-Btria} by duality to $H_D^{1+s_0}(\Omega)\times H_D^{1-s_0}(\Omega)$ and the stability of $\Pi_\tria$ at both orders yields
  \begin{align*}
    \abs{(B_\tria v)(z)}
    &\lesssim
    \norm{v}_{H_D^{1+s_0}(\Omega)}
    \norm{z}_{H_D^{1-s_0}(\Omega)} \quad \text{for all }\text{$v\in H_D^{1+s_0}(\Omega)$ and $z\in H_D^{1-s_0}(\Omega)$}.
  \end{align*}
 By symmetry, the same estimate holds with the two smoothness orders interchanged and so, with $X_-\coloneqq H_D^{1-s_0}(\Omega)$ and $X_+\coloneqq H_D^{1+s_0}(\Omega)$, we have obtained uniform boundedness of $B_\tria\colon X_-\to X_+^*$ and $B_\tria\colon X_+\to X_-^*$.  The complex interpolation identities of~\cite[Thm.~1.2, (c)]{BechtelEgert2019}, together with duality, give for $0<\theta<1$
  \begin{align*}
    [X_-,X_+]_\theta
    &=
    H_D^{1+(2\theta-1)s_0}(\Omega) \qquad\text{and}\qquad
    [X_+^*,X_-^*]_\theta
    =
    \bigl(H_D^{1-(2\theta-1)s_0}(\Omega)\bigr)^*.
  \end{align*}
  At $\theta=1/2$, these spaces are $H^1_0(\Omega)$ and $H^{-1}(\Omega)$, respectively, and $B_\tria$ is uniformly invertible by~\eqref{eq:B-energy-inverse}. Applying the quantitative Shneiberg theorem, we obtain some value $\tau\in(0,s_0)$, independent of $\tria$, such that
  \begin{align}\label{eq:B-Hs-inverse}
    B_\tria\colon H_D^{1+s}(\Omega)
    \to
    \bigl(H_D^{1-s}(\Omega)\bigr)^*
    \quad\text{is uniformly invertible for all }
    \abs{s}\leq\tau.
  \end{align}

  \textit{Step 4 (Transfer to the Ritz projection).}
  Since $\Pi_\tria R_\tria v=R_\tria v$ and $\Pi_\tria z\in V(\tria)$, the definition of the Ritz projection gives for all $v\in W^{1,1}(\Omega)$ and $z\in W^{1,1}_0(\Omega)$
  \begin{align}\label{eq:B-Ritz-identity}
    (B_\tria R_\tria v)(z)
    &=
    \int_\Omega\nabla R_\tria v\cdot\nabla\Pi_\tria z\dx
    =
    \int_\Omega\nabla v\cdot\nabla\Pi_\tria z\dx
    \eqqcolon
    \ell_v(z).
  \end{align}
  Let $1<p<\infty$ with $\abs{1/p-1/2}\leq\epsilon$. By the $W^{1,p'}(\Omega)$-stability of $\Pi_\tria$, we have
  \begin{align*}
    \abs{\ell_v(z)}
    &\leq
    \norm{\nabla v}_{L^p(\Omega)}
    \norm{\nabla\Pi_\tria z}_{L^{p'}(\Omega)}
    \lesssim
    \norm{\nabla v}_{L^p(\Omega)}
    \norm{\nabla z}_{L^{p'}(\Omega)}.
  \end{align*}
This shows $\norm{\ell_v}_{W^{-1,p}(\Omega)}
    \lesssim
    \norm{\nabla v}_{L^p(\Omega)}$ for all $v\in W^{1,p}(\Omega)$. Hence, combining~\eqref{eq:B-Ritz-identity} with the uniform invertibility~\eqref{eq:BT-W1p-invertible} proves \ref{itm:Ritz-W1p}.

  Next, let $0\leq s\leq\tau < 1/2$ and let $v\in H_D^{1+s}(\Omega)$. Since the gradients of finite element functions are piecewise polynomials, Lemma~\ref{lem:localized-Wsq} implies $    V(\tria)\subset H_D^{1+s}(\Omega)$.
  In particular, $R_\tria v\in H_D^{1+s}(\Omega)$. For $z\in H_D^{1-s}(\Omega)$, the stability of $\Pi_\tria$ and the duality between $H^s(\Omega)$ and $H^{-s}(\Omega)$ give
  \begin{align*}
    \abs{\ell_v(z)}
    &\lesssim
    \norm{\nabla v}_{H^s(\Omega)}
    \norm{\nabla\Pi_\tria z}_{H^{-s}(\Omega)}
    \lesssim
    \norm{v}_{H^{1+s}(\Omega)}
    \norm{z}_{H_D^{1-s}(\Omega)}.
  \end{align*}
  This shows 
  \begin{align*}
    \norm{\ell_v}_{(H_D^{1-s}(\Omega))^*}
    \lesssim
    \norm{v}_{H^{1+s}(\Omega)}.
  \end{align*}
  Since $\Pi_\tria z\in V(\tria)$, the identity~\eqref{eq:B-Ritz-identity} also holds for $z\in H_D^{1-s}(\Omega)$, with $B_\tria$ interpreted through the duality extension of Step~3. The uniform invertibility~\eqref{eq:B-Hs-inverse} therefore yields \ref{itm:Ritz-Hs}; that is,
  \begin{align*}
    \norm{R_\tria v}_{H^{1+s}(\Omega)}
    &\lesssim
    \norm{\ell_v}_{(H_D^{1-s}(\Omega))^*}
    \lesssim
    \norm{v}_{H^{1+s}(\Omega)}.
  \end{align*}

  \textit{Step 5 (Interpolation).}
  Choose $r\in(1,2)$ satisfying $\abs{\frac1r-\frac12}\leq\epsilon$, and let $\tau$ be as in \ref{itm:Ritz-Hs}. Set
  \begin{align*}
    a&\coloneqq \frac dr-\frac d2,
    &
    \lambda&\coloneqq\frac{a}{a+\tau},
    &
    \frac1q&\coloneqq\frac{1-\lambda}{r}+\frac{\lambda}{2},
    &
    \sigma&\coloneqq\lambda\tau.
  \end{align*}
  Then we have $r<q<2$ as well as 
  \begin{align*}
    \sigma
    =\lambda\tau
    =(1-\lambda)a
    =\frac dq-\frac d2\qquad\text{and}
    \qquad
    0<\sigma<\tau<\frac12<\frac1q.
  \end{align*}
  In particular, this shows $\sigma q<1$.
  By \ref{itm:Ritz-W1p} and \ref{itm:Ritz-Hs}, the Ritz projection is uniformly bounded on both $H_D^{1,r}(\Omega)$ and $H_D^{1+\tau,2}(\Omega)$. Both endpoints are Bessel potential spaces, and their smoothness orders are distinct. The classical real interpolation formula for Bessel potential spaces~\cite[Sec.~2.4.2]{Triebel1978}, together with its trace-preserving version~\cite[Thm.~1.2, (d)]{BechtelEgert2019}, gives
  \begin{align*}
    \bigl(H_D^{1,r}(\Omega),H_D^{1+\tau,2}(\Omega)\bigr)_{\lambda,q}
    &=
    W_D^{1+\sigma,q}(\Omega)
   =
    W^{1+\sigma,q}(\Omega)\cap W^{1,q}_0(\Omega).
  \end{align*}
  Consequently, we have for all $v\in W^{1+\sigma,q}(\Omega)\cap W^{1,q}_0(\Omega)$ the bound
  \begin{align}\label{eq:Ritz-W1sigmaq}
    \norm{R_\tria v}_{W^{1+\sigma,q}(\Omega)}
    \lesssim \norm{v}_{W^{1+\sigma,q}(\Omega)}.
  \end{align}
Moreover, such functions $v$ satisfy $\int_\Omega\nabla v\dx=0$. The fractional Poincar\'e inequality applied to $\nabla v$ and the usual Poincar\'e inequality therefore yield
  \begin{align*}
    \norm{v}_{W^{1+\sigma,q}(\Omega)}
    \lesssim
    \abs{\nabla v}_{W^{\sigma,q}(\Omega)}.
  \end{align*}
  Combining this with~\eqref{eq:Ritz-W1sigmaq}, we obtain
  \begin{align*}
    \abs{\nabla R_\tria v}_{W^{\sigma,q}(\Omega)}
    \lesssim
    \norm{R_\tria v}_{W^{1+\sigma,q}(\Omega)}
    \lesssim
    \norm{v}_{W^{1+\sigma,q}(\Omega)}
    \lesssim
    \abs{\nabla v}_{W^{\sigma,q}(\Omega)}.
  \end{align*}
  This proves \ref{itm:Ritz-Wsq} and concludes the proof.
\end{proof}


\subsection{Sparse projection errors}
\label{subsec:SparseProjectionError}

The difference between Galerkin solutions on two nested meshes may be globally supported. The key observation is that the error of any uniformly energy-stable projection of the finer solution onto the coarse space has a comparable energy norm; see \cite[Lemma~3.4]{DieningKreuzerStevenson2016}. Related decompositions into coarse and detail components are used in multigrid and hierarchical basis methods; cf.~\cite[Secs.~3.4 and~6.2]{Xu1992}. Choosing a Scott--Zhang projection adapted to the pair of meshes, as in \cite[Lemma~3.5]{DieningKreuzerStevenson2016}, makes this projection error vanish on the common elements. We thus obtain a sparse representative of the Galerkin difference with a comparable energy norm.
We formulate this construction for arbitrary Galerkin-orthogonal functions and combine the resulting Ritz reconstruction with the fractional stability established in the preceding subsection. 

Let $\tria,\tria_\star\in\meshes$ with $\tria\leq\tria_\star$ and consider the difference
\begin{align*}
  w\coloneqq u_{\tria_\star}-u_\tria\in V(\tria_\star).
\end{align*}
Galerkin orthogonality yields
\begin{align}\label{eq:Ortho}
  \int_\Omega\nabla w\cdot\nabla v\dx=0
  \qquad\text{for all }v\in V(\tria).
\end{align}
Our proof uses a Scott--Zhang-type projection~\cite{ScottZhang1990}
\begin{align*}
  \Pi_\tria^{\tria_\star}\colon V(\tria_\star)\to V(\tria)
\end{align*}
whose averaging simplices are adapted to the pair of meshes following~\cite[Lemma~3.5]{DieningKreuzerStevenson2016}. At Lagrange nodes on the mesh skeleton inside $\Omega$, we average over an incident coarse facet, choosing a facet of a common element whenever possible. At nodes in the interior of a coarse element, we average over that element; boundary coefficients are set to zero. The operator is a projection onto $V(\tria)$ and satisfies, with hidden constant depending only on the polynomial degree and shape regularity,
\begin{align}\label{eq:SZStability}
  \|\nabla\Pi_\tria^{\tria_\star}v\|_{L^2(\Omega)}
  &\lesssim
  \|\nabla v\|_{L^2(\Omega)}
  \qquad \text{for all }v\in V(\tria_\star).
\end{align}
It preserves finite element functions on common elements in the sense that
\begin{align}\label{eq:SZCommon}
  (\Pi_\tria^{\tria_\star}v)|_T
  &=
  v|_T
  \qquad \text{for all }v\in V(\tria_\star)
  \text{ and }T\in\tria\cap\tria_\star.
\end{align}

\begin{lemma}[Sparse projection error]
  \label{lem:SparseProjectionError}
  Let $\tria,\tria_\star\in\meshes$ with $\tria \leq \tria_\star$ and let $w\in V(\tria_\star)$ satisfy \eqref{eq:Ortho}. Then the sparse projection error
  \begin{align*}
    z\coloneqq w-\Pi_\tria^{\tria_\star}w\in V(\tria_\star)
  \end{align*}
  has the following properties:
  \begin{enumerate}[itemsep=0.3em]
  \item\label{itm:reconstruction}
    \textup{\textbf{Ritz reconstruction.}}
    $w=(\operatorname{Id}-R_\tria)z$;

  \item\label{itm:Equivalence}
    \textup{\textbf{Equivalence.}}
    $\|\nabla z\|_{L^2(\Omega)}\eqsim\|\nabla w\|_{L^2(\Omega)}$;

  \item\label{itm:Sparsity}
    \textup{\textbf{Sparsity.}} We have $ \set{\nabla z\neq0}
      \subseteq
        \bigcup(\tria\setminus\tria_\star)
        =
        \bigcup(\tria_\star\setminus\tria)$\\ \hphantom{\textup{\textbf{Sparsity. }}}and $\#\set{T\in\tria_\star\colon\nabla z|_T\neq0}
      \leq
        2(\#\tria_\star-\#\tria)$.
  \end{enumerate}
\end{lemma}

\begin{proof}
  The orthogonality \eqref{eq:Ortho} implies $R_\tria w=0$, while the projection property of the Ritz map gives $R_\tria\Pi_\tria^{\tria_\star}w = \Pi_\tria^{\tria_\star}w$.  Therefore, \ref{itm:reconstruction} follows by
  \begin{align*}
    (\operatorname{Id}-R_\tria)z
   &  =
    (\operatorname{Id}-R_\tria)
    \bigl(w-\Pi_\tria^{\tria_\star}w\bigr)
    =
    w.
  \end{align*}
  The $W^{1,2}(\Omega)$-stability of $\Pi_\tria^{\tria_\star}$ and $R_\tria$ gives \ref{itm:Equivalence}. By \eqref{eq:SZCommon}, the function $z$ vanishes on every common element $T\in\tria\cap\tria_\star$, which proves the first assertion in~\ref{itm:Sparsity}. The second assertion follows from
  \begin{align*}
    \#\bigl\{T\in\tria_\star\colon\nabla z|_T\neq0\bigr\}
    &\leq
    \#(\tria_\star\setminus\tria)
    =
    \#\tria_\star-\#\tria+\#(\tria\setminus\tria_\star)
    \leq
    2(\#\tria_\star-\#\tria).
  \end{align*}
  This concludes the proof.
\end{proof}
The sparsity of the projection error $z$, quantified in \ref{itm:Sparsity}, will produce the factor $(\#\tria_\star-\#\tria)^{1/q-1/2}$ in \eqref{eq:EstimatorChangeCriterion}. The Ritz reconstruction \ref{itm:reconstruction}, together with the fractional stability of the Ritz and Scott--Zhang projections, yields the following fractional norm equivalence.
\begin{lemma}[Fractional norm equivalence of the sparse projection error]
  \label{lem:fractionalSparseProjection}
  Let $\tria,\tria_\star\in\meshes$ with $\tria \leq \tria_\star$ and let $w\in V(\tria_\star)$ satisfy \eqref{eq:Ortho}. For the sparse projection error $z\coloneqq w-\Pi_\tria^{\tria_\star}w$ and the exponents $q\in(1,2)$ and $\sigma=d/q- d/2>0$ from Lemma~\ref{lem:PoissonCriticalRitz}, we have
  \begin{align*}
    \abs{\nabla w}_{W^{\sigma,q}(\Omega)}
    \eqsim
    \abs{\nabla z}_{W^{\sigma,q}(\Omega)}.
  \end{align*}
\end{lemma}

\begin{proof}
  Since $w=(\operatorname{Id}-R_\tria)z$, Lemma~\ref{lem:PoissonCriticalRitz} gives
  \begin{align*}
    \abs{\nabla w}_{W^{\sigma,q}(\Omega)}
    \lesssim
    \abs{\nabla z}_{W^{\sigma,q}(\Omega)}.
  \end{align*}
  Conversely, $z=w-\Pi_\tria^{\tria_\star}w$. On $V(\tria_\star)$, the adapted projection $\Pi_\tria^{\tria_\star}$ agrees with a full trace-preserving Scott--Zhang operator having the same interior averaging functionals and boundary-facet averages at boundary nodes. Hence Lemma~\ref{lem:SZFractionalStability}, whose constants are uniform over these choices, gives
  \begin{align*}
    \abs{\nabla z}_{W^{\sigma,q}(\Omega)}
    & \lesssim
    \abs{\nabla w}_{W^{\sigma,q}(\Omega)}.\qedhere
  \end{align*}
\end{proof}

The fractional norm equivalence above allows us to estimate the Galerkin difference through its sparse projection error. The following Bernstein estimate quantifies the gain from this sparsity.  For $\tria\in\meshes$, let $\mathbb P_k(\tria)$ denote the space of functions that are polynomials of degree at most $k$ on each $T\in\tria$.

\begin{lemma}[Sparse Bernstein estimate]
  \label{lem:inverseEst}
  Let $1<q<2$ and $\sigma=d/q- d/2>0$ with $\sigma q<1$ and let $\tria \in \meshes$.
  \begin{enumerate}[itemsep=0.3em]
  \item\label{itm:1Bernstein} For all $g\in\mathbb P_{k-1}(\tria)$, we have the Bernstein estimate
    \begin{align*}
      \abs{g}_{W^{\sigma,q}(\Omega)}
      \lesssim
      \Big(
        \sum_{T\in\tria}
        \norm{g}_{L^2(T)}^q
      \Big)^{1/q}.
    \end{align*}
  \item\label{itm:2Bernstein} If, in addition, $g$ vanishes on all but at most $M\in\mathbb N$ elements, then we have
    \begin{align*}
      \abs{g}_{W^{\sigma,q}(\Omega)}
      \lesssim
      M^{\frac 1q- \frac 12}
      \norm{g}_{L^2(\Omega)}.
    \end{align*}
  \end{enumerate}
  The hidden constants depend on the polynomial degree, $q$, the shape regularity, and the distance of $\sigma q$ from one.
\end{lemma}

\begin{proof}
Let $g\in\mathbb P_{k-1}(\tria)$.
  Since $0<\sigma< 1/q$, Lemma~\ref{lem:localized-Wsq} gives
  \begin{align*}
    \abs{g}_{W^{\sigma,q}(\Omega)}^q
    \lesssim
    \sum_{T\in\tria}
    \big(
      \abs{g}_{W^{\sigma,q}(T)}^q
      +
      h_T^{-\sigma q}\norm{g}_{L^q(T)}^q
    \big).
  \end{align*}
  The inverse estimate $\abs{g}_{W^{\sigma,q}(T)}^q
    \lesssim h_T^{-\sigma q}\norm{g}_{L^q(T)}^q$ and H\"older's inequality yield due to the scaling $\sigma=
  d/q- d/2$, cf.~Remark~\ref{rem:scaling},
  \begin{align*}
    \abs{g}_{W^{\sigma,q}(\Omega)}^q
    &\lesssim
    \sum_{T\in\tria}
    h_T^{-\sigma q}
    \norm{g}_{L^q(T)}^q
  \lesssim
    \sum_{T\in\tria}
    h_T^{-\sigma q}
    \big(
      |T|^{\frac 1q-\frac 12}\norm{g}_{L^2(T)}
    \big)^q
    \eqsim
    \sum_{T\in\tria}
    \norm{g}_{L^2(T)}^q.
  \end{align*}
  This proves \ref{itm:1Bernstein}. If at most $M$ terms are nonzero, H\"older's inequality for the finite sum implies \ref{itm:2Bernstein} in the sense that
  \begin{align*}
    \Big(
      \sum_{T\in\tria}\norm{g}_{L^2(T)}^q
    \Big)^{ 1/q}
    &\leq
    M^{\frac 1q- \frac 12}
    \norm{g}_{L^2(\Omega)}.\qedhere
  \end{align*}
\end{proof}


\subsection{Control of residual estimator changes}
\label{subsec:ResidualEstimatorChanges}

We now prove the estimator-change estimate in \eqref{eq:EstimatorChangeCriterion}. For $\tria\in\meshes$, $v\in V(\tria)$, and $T\in\tria$, define the homogeneous residual indicator
\begin{align*}
  \eta_\tria^{\rm hom}(v,T)^2
  \coloneqq
  h_T^2\|\Delta_hv\|_{L^2(T)}^2
  +
  h_T
  \sum_{\substack{F\subset\partial T\\F\in\mathcal F_\tria^{\mathrm{int}}}}
  \|\llbracket\nabla v\cdot n_F\rrbracket\|_{L^2(F)}^2.
\end{align*}
In other words, $\eta_\tria^{\rm hom}(v,T)$ is obtained from the residual indicator \eqref{eq:PoissonEstimator} by replacing $u_\tria$ by $v$ and setting $f=0$. The homogeneous indicator naturally occurs when comparing two residual estimators, since the data term $f$ cancels.

\begin{lemma}[Change of local error indicators]
  \label{lem:PoissonResidualChange}
  Let $\tria,\tria_\star\in\meshes$ with $\tria\leq\tria_\star$ and set $w\coloneqq u_{\tria_\star}-u_\tria$. Then we have for all $T\in\tria\cap\tria_\star$
  \begin{align*}
    |\eta_\tria(T)-\eta_{\tria_\star}(T)|
    \leq
    \eta_{\tria_\star}^{\rm hom}(w,T).
  \end{align*}
\end{lemma}

\begin{proof}
  Let $T\in\tria\cap\tria_\star$. Since $T$ is unrefined, its local mesh size and facets agree in both triangulations. The reverse triangle inequality applied to the vector consisting of the volume and jump residuals therefore gives
  \begin{align*}
    |\eta_{\tria_\star}(T)-\eta_\tria(T)|^2
    &\leq
    h_T^2\|\Delta_h(u_{\tria_\star}-u_\tria)\|_{L^2(T)}^2
    \\
    &\quad+
    h_T\!\!
    \sum_{ F\subset\partial T\cap \mathcal F_{\tria}^{\mathrm{int}}}
    \|\llbracket\nabla(u_{\tria_\star}-u_\tria)\cdot n_F\rrbracket\|_{L^2(F)}^2
  =
    \eta_{\tria_\star}^{\rm hom}(w,T)^2.\qedhere
  \end{align*}
\end{proof}

The next lemma bounds the homogeneous residual indicator by the fractional Sobolev seminorm of the gradient.

\begin{lemma}[Fractional control of the homogeneous residual]
  \label{lem:PoissonResidualExtraction}
  Let $q\in(1,2)$ and $\sigma\coloneqq d/q-d/2>0$ with $\sigma q<1$. Then we have
  \begin{align*}
    \sum_{T\in\tria}
    \eta_\tria^{\rm hom}(v,T)^q
    \lesssim
    |\nabla v|_{W^{\sigma,q}(\Omega)}^q\qquad\text{for all } \tria\in\meshes\text{ and }v\in V(\tria).
  \end{align*}
\end{lemma}

\begin{proof}
Let $T\in\tria$ and $v\in V(\tria)$. We denote the integral mean by $
    \mean{\nabla v}_T
    \coloneqq
    \dashint_T\nabla v\dx$.
  Since $\nabla v-\mean{\nabla v}_T$ is a polynomial on~$T$, inverse estimates, the fractional Poincar\'e inequality, and the scaling $\sigma= d/q-d/2$ give
  \begin{align*}
    h_T\|\Delta_hv\|_{L^2(T)}
    &\lesssim
    h_T\|\nabla^2v\|_{L^2(T)}
    \lesssim
    \|\nabla v-\mean{\nabla v}_T\|_{L^2(T)}
    \\
    &\lesssim
    h_T^{-\sigma}
    \|\nabla v-\mean{\nabla v}_T\|_{L^q(T)}
    \lesssim
    |\nabla v|_{W^{\sigma,q}(T)}.
  \end{align*}
For $F=T\cap T'\in\mathcal F_\tria^{\mathrm{int}}$, set $\omega_F\coloneqq T\cup T'$ and define $    \mean{\nabla v}_{\omega_F}
    \coloneqq
    \dashint_{\omega_F}\nabla v\dx$.
  By shape regularity, the diameters of $T$ and $T'$ are comparable. Trace and inverse estimates followed by inverse and fractional Poincar\'e inequalities yield
  \begin{align*}
    h_T^{1/2}
    \|\llbracket\nabla v\cdot n_F\rrbracket\|_{L^2(F)}
    &\lesssim
    \|\nabla v-\mean{\nabla v}_{\omega_F}\|_{L^2(\omega_F)}
    \\
    &\lesssim
    h_T^{-\sigma}
    \|\nabla v-\mean{\nabla v}_{\omega_F}\|_{L^q(\omega_F)}
    \lesssim
    |\nabla v|_{W^{\sigma,q}(\omega_F)}.
  \end{align*}
  Since $q<2$, we have $(a^2 + b^2)^{q/2} \leq a^q + b^q$ for nonnegative numbers $a,b\geq 0$.  Hence, using that every simplex has only finitely many facets and that the patches $\omega_F$ have uniformly bounded overlap, we obtain
  \begin{align*}
    \sum_{T\in\tria}
    \eta_\tria^{\rm hom}(v,T)^q
    &\lesssim
    \sum_{T\in\tria}
    |\nabla v|_{W^{\sigma,q}(T)}^q
    +
    \sum_{F\in\mathcal F_\tria^{\mathrm{int}}}
    |\nabla v|_{W^{\sigma,q}(\omega_F)}^q
\lesssim
    |\nabla v|_{W^{\sigma,q}(\Omega)}^q.\qedhere
  \end{align*}
\end{proof}

\begin{lemma}[Estimator-change estimate]
  \label{lem:PoissonEstimatorChange}
  Let $q\in(1,2)$ be the exponent of Lemma~\ref{lem:PoissonCriticalRitz} and set $\alpha\coloneqq1/q-1/2>0$. We have, for all $\tria,\tria_\star\in\meshes$ with $\tria\leq\tria_\star$,
  \begin{align}\label{eq:PoissonEstimatorChange}
    \Big(
      \sum_{T\in\tria\cap\tria_\star}
      |\eta_\tria(T)-\eta_{\tria_\star}(T)|^q
    \Big)^{1/q}
    \lesssim
    (\#\tria_\star-\#\tria)^\alpha
    \delta(\tria,\tria_\star).
  \end{align}
\end{lemma}

\begin{proof}
  The assertion is trivial if $\tria=\tria_\star$, so assume $\tria<\tria_\star$ and set $w\coloneqq u_{\tria_\star}-u_\tria$. Galerkin orthogonality gives \eqref{eq:Ortho}. Let $z\in V(\tria_\star)$ be the sparse projection error of~$w$ from Lemma~\ref{lem:SparseProjectionError}.
  Lemma~\ref{lem:fractionalSparseProjection}, the equivalence and sparsity properties of Lemma~\ref{lem:SparseProjectionError}, and Lemma~\ref{lem:inverseEst}, applied componentwise to $\nabla z$, yield with \eqref{eq:PoissonDiscreteDifference}
  \begin{align*}
    |\nabla w|_{W^{\sigma,q}(\Omega)}
    &\eqsim
    |\nabla z|_{W^{\sigma,q}(\Omega)}
   \lesssim
    (\#\tria_\star-\#\tria)^\alpha
    \|\nabla z\|_{L^2(\Omega)}
    \\
    &\lesssim
    (\#\tria_\star-\#\tria)^\alpha
    \|\nabla w\|_{L^2(\Omega)}
    \leq
    \sqrt2\,
    (\#\tria_\star-\#\tria)^\alpha
    \delta(\tria,\tria_\star).
  \end{align*}
  Combining Lemmas~\ref{lem:PoissonResidualChange} and \ref{lem:PoissonResidualExtraction} gives
  \begin{align*}
    \Big(
      \sum_{T\in\tria\cap\tria_\star}
      |\eta_\tria(T)-\eta_{\tria_\star}(T)|^q
    \Big)^{1/q}
    &\leq
    \Big(
      \sum_{T\in\tria_\star}
      \eta_{\tria_\star}^{\rm hom}(w,T)^q
    \Big)^{1/q}
\lesssim
    |\nabla w|_{W^{\sigma,q}(\Omega)}.
  \end{align*}
  Combining these bounds proves \eqref{eq:PoissonEstimatorChange}.
\end{proof}

\begin{proof}[Proof of Theorem~\ref{thm:PoissonEstimatorConcentration}]
  Lemma~\ref{lem:PoissonEstimatorChange} verifies the estimator-change criterion \eqref{eq:EstimatorChangeCriterion}. Hence Lemma~\ref{lem:EstimatorChangeCriterion} proves estimator concentration. The rate-optimality assertions then follow from Theorem~\ref{thm:ClassOptimalityAllTheta} and the equivalence $\err(\tria)\eqsim\eta_\tria$ established in Section~\ref{subsec:PoissonSetup}.
\end{proof}


\appendix

\section{Fractional stability of the Scott--Zhang projection}\label{app:SZFractionalStability}

Let $\Omega\subset\RR^d$ with $d\geq2$ be a bounded polyhedral Lipschitz domain, and let $\tria\in\meshes$ be a conforming, shape-regular simplicial triangulation. For a fixed polynomial degree $k\in\mathbb N$, let $\Pi_\tria$ denote the Scott--Zhang projection onto the full continuous Lagrange space, with the averaging functionals associated with boundary nodes supported on boundary facets; see \cite{ScottZhang1990}. This operator reproduces affine functions and preserves finite element boundary traces. In particular, it maps $W^{1,q}_0(\Omega)$ into $V(\tria)$.
We start with a localization estimate for the $W^{s,q}$-seminorm. Related localization results in the Hilbert space setting are given in \cite[Thm~4.1]{AinsworthMcLeanTran1999} and \cite[Sec.~3.1]{FeischlFuehrerHeuerKarkulikPraetorius2015}. For completeness, we give the short proof for $0<s<1/q$, based on Grisvard's zero extension result \cite[Corollary~1.4.4.5]{Grisvard1985}.
\begin{lemma}[Localization of the fractional Sobolev seminorm]  \label{lem:localized-Wsq}
  Let $1<q<\infty$ and $0<s< 1/q$, and let $\tria$ be a finite, shape-regular simplicial triangulation of $\Omega$. Then every $g\in L^q(\Omega)$ with $g|_T\in W^{s,q}(T)$ for all $T\in\tria$ belongs to $W^{s,q}(\Omega)$, and
  \begin{align*}
    \abs{g}_{W^{s,q}(\Omega)}^q \lesssim \sum_{T\in\tria} \big(\abs{g}_{W^{s,q}(T)}^q +h_T^{-sq}\norm{g}_{L^q(T)}^q \big).
  \end{align*}
  The hidden constant depends only on $d$, $s$, $q$, and the shape-regularity of $\tria$. The statement extends to vector-valued functions.
\end{lemma}
\begin{proof}
  For $T\in\tria$ and $f\in W^{s,q}(T)$ with $1<q<\infty$ and $0<s< 1/q$, let $E_Tf$ denote its extension by zero to $\RR^d$. Grisvard's zero extension result \cite[Cor.~1.4.4.5]{Grisvard1985}, see also \cite[Thm.~3.19]{ErnGuermond2021}, applied on the reference simplex and scaled to $T$ gives
  \begin{align}
    \abs{E_Tf}_{W^{s,q}(\RR^d)}^q
    &\lesssim \abs{f}_{W^{s,q}(T)}^q + h_T^{-sq}\norm{f}_{L^q(T)}^q.
      \label{eq:grisvard}
  \end{align}
  This estimate holds also for vector-valued functions. Let $g\in L^q(\Omega)$ with $g|_T\in W^{s,q}(T)$ and $g_T \coloneqq E_T(g|_T)$ for all $T\in \tria$. Then we have $g=\sum_{T\in\tria}g_T$ almost everywhere in $\Omega$. For almost every pair $(x,y)\in\RR^d\times\RR^d$, at most two differences $g_T(x)-g_T(y)$ are nonzero. Consequently, using~\eqref{eq:grisvard}, we obtain
  \begin{align*}
    \abs{g}_{W^{s,q}(\Omega)}^q &\leq 2^{q-1}\sum_{T\in\tria} \abs{g_T}_{W^{s,q}(\RR^d)}^q \lesssim \sum_{T\in\tria} \big( \abs{g}_{W^{s,q}(T)}^q + h_T^{-sq}\norm{g}_{L^q(T)}^q \big).
  \end{align*}
  The right-hand side is finite, proving $g \in W^{s,q}(\Omega)$ and the estimate.
\end{proof}

Fractional stability of the trace-preserving Scott--Zhang projection for orders between $1/q$ and $1$ was established by Ciarlet~\cite[Cor.~3.4 and Sec.~4]{Ciarlet2013} for dimensions $d\leq3$. The argument extends to arbitrary fixed dimension, since the required scaled trace estimates, local approximation estimates, and bounded overlap of vertex patches remain valid for shape-regular simplicial triangulations. For Scott--Zhang-type operators based on volume averaging, Faustmann, Melenk, and Parvizi~\cite[Thm.~2.2 and Cor.~3.7]{FaustmannMelenkParvizi2021} obtained higher-order stability results, including endpoint estimates in Besov spaces. The following lemma combines the lower-order result by Ciarlet with a direct proof of $W^{1+s,q}$-stability for the trace-preserving Scott--Zhang projection, based on the localization estimate above.

\begin{lemma}[Fractional trace-preserving Scott--Zhang stability]
  \label{lem:SZFractionalStability}
  Let $1<q<\infty$ and $1/q < t < 1 + 1/q$ with $t \neq 1$.
  Then, we have for all $v\in W^{t,q}(\Omega)$ the bound
  \begin{align}\label{eq:SZFractionalStability}
    \norm{\Pi_\tria v}_{W^{t,q}(\Omega)}
    \lesssim
    \norm{v}_{W^{t,q}(\Omega)}.
  \end{align}
  If $t=1+s$ with $0<s<1/q$, then we have additionally
  \begin{align}\label{eq:SZFractionalGradientStability}
    \abs{\nabla\Pi_\tria v}_{W^{s,q}(\Omega)}
    \lesssim
    \abs{\nabla v}_{W^{s,q}(\Omega)}.
  \end{align}
  The constants depend only on $d$, $k$, $q$, $t$, the shape regularity of $\tria$, and, for the full norm estimate, $\Omega$. They are uniform with respect to the admissible choices of averaging simplices.
  In particular, $\Pi_\tria$ preserves homogeneous Dirichlet traces in both ranges.
\end{lemma}

\begin{proof}
  For $1/q<t<1$, the global approximation estimate of~\cite[Cor.~3.4 and Sec.~4]{Ciarlet2013}, applied with equal regularity and approximation orders, gives
  \begin{align*}
    \norm{v-\Pi_\tria v}_{W^{t,q}(\Omega)}
    \lesssim
    \abs{v}_{W^{t,q}(\Omega)}.
  \end{align*}
  The triangle inequality proves~\eqref{eq:SZFractionalStability} in this range.
  Set $e\coloneqq v-\Pi_\tria v$. For $T\in\tria$, let $\omega_T$ be the patch of elements sharing a vertex with $T$, and recall that $h_T=|T|^{1/d}$. We choose an affine function $p_T$ with $\nabla p_T=\mean{\nabla v}_{\omega_T} \coloneqq |\omega_T|^{-1}\int_{\omega_T}\nabla v\dx$. Affine reproduction, local gradient stability, a polynomial inverse estimate, and the fractional Poincar\'e inequality yield
  \begin{align}\label{eq:SZLocalFractionalError}
    \begin{aligned}
      \lefteqn{h_T^{-s}\norm{\nabla e}_{L^q(T)} + \abs{\nabla e}_{W^{s,q}(T)}} \qquad &\\
      &\lesssim h_T^{-s}\norm{\nabla (v- p_T)}_{L^q(T)} +\abs{\nabla (v-p_T)}_{W^{s,q}(T)}\\
      &\quad + h_T^{-s}\norm{\nabla \Pi_\tria (v-p_T)}_{L^q(T)} + \abs{\nabla \Pi_\tria(v-p_T)}_{W^{s,q}(T)} \\
      &\lesssim h_T^{-s}\norm{\nabla v-\mean{\nabla v}_{\omega_T}}_{L^q(T)} + \abs{\nabla v}_{W^{s,q}(T)} + h_T^{-s}\norm{\nabla \Pi_\tria (v-p_T)}_{L^q(T)} \\
      &\lesssim \abs{\nabla v}_{W^{s,q}(T)} + h_T^{-s}\norm{\nabla(v-p_T)}_{L^q(\omega_T)} \lesssim \abs{\nabla v}_{W^{s,q}(\omega_T)}.
    \end{aligned}
  \end{align}
  By Lemma~\ref{lem:localized-Wsq} we have
  \begin{align*}
   \abs{\nabla e}_{W^{s,q}(\Omega)}^q &\lesssim \sum_{T\in\tria} \big(\abs{\nabla e}_{W^{s,q}(T)}^q + h_T^{-sq}\norm{\nabla e}_{L^q(T)}^q \big).
  \end{align*}
  This estimate, \eqref{eq:SZLocalFractionalError}, and the bounded overlap of the patches yield
  \begin{align*}
    \abs{\nabla e}_{W^{s,q}(\Omega)}^q
    &\lesssim
      \sum_{T\in\tria}
      \abs{\nabla v}_{W^{s,q}(\omega_T)}^q \lesssim \abs{\nabla v}_{W^{s,q}(\Omega)}^q.
  \end{align*}
  The triangle inequality proves~\eqref{eq:SZFractionalGradientStability}. Combining this estimate with the $W^{1,q}(\Omega)$-stability of $\Pi_\tria$ proves~\eqref{eq:SZFractionalStability} for $t=1+s$.
\end{proof}


\subsection*{Acknowledgment}
This work is part of a longer-term project of the authors on adaptive finite element methods. During the final stage of the present work, an extended interactive discussion between Johannes Storn and OpenAI's ChatGPT, motivated in particular by our recent results on $L^2$ and Ritz projection stability, led to an initial framework for the proof. Since the mathematical ideas developed through a long sequence of questions, suggestions, and revisions, we do not attempt to separate the individual contributions of the human and AI participants to this stage. The authors subsequently verified, corrected, generalized, and substantially rewrote the arguments and take full responsibility for all content in this paper.

\printbibliography
\end{document}